\documentclass[reqno]{amsart}
\usepackage{hyperref}

\def \R{\mathbb{R}}

\def \E{\mathbb{E}}

\def \bf{\textbf}
\def \it{\textit}

\def \ni {\noindent}

\vspace{9mm}
\begin{document}

\title[  Controllability  stochastic    integro-differential equations with fBm]
{CONTROLLABILITY OF  NEUTRAL STOCHASTIC  FUNCTIONAL
INTEGRO-DIFFERENTIAL EQUATIONS DRIVEN BY    FRACTIONAL BROWNIAN
MOTION}

\author[E. Lakhel ]{ El Hassan  Lakhel  }
\maketitle
\begin{center}{ National School of
Applied Sciences, Cadi Ayyad University, 46000 Safi ,\\
Morocco}\end{center}

 \vspace{0.8cm}
\begin{abstract}  This paper  focuses on controllability results of
   stochastic delay partial functional  integro-differential equations
   perturbed by     fractional Brownian motion with Hurst parameter $H\in(\frac{1}{2},1)$.
    Sufficient conditions  are established using the theory
   of resolvent operators developed  by R. Grimmer in \cite{gri82} combined with a fixed point approach
   for achieving the required result.
   An  example is provided  to illustrate the  theory.
\end{abstract}
\medskip

\footnotetext[1]{Corresponding author, Email : e.lakhel@uca.ma}

\medskip
 \ni {\bf {Keywords:}} Neutral stochastic partial  integro-differential equations,
  resolvent operators,   fractional Brownian motion.\\

 \ni \bf{ AMS Subject Classification:} 60G18;   60G22; 60H20.


\maketitle \numberwithin{equation}{section}
\newtheorem{theorem}{Theorem}[section]
\newtheorem{lemma}[theorem]{Lemma}
\newtheorem{proposition}[theorem]{Proposition}
\newtheorem{definition}[theorem]{Definition}
\newtheorem{example}[theorem]{Example}
\newtheorem{remark}[theorem]{Remark}
\allowdisplaybreaks

\ni\section{Introduction}

The noise or perturbations of a system are typically modeled by a
Brownian motion as such a process is Gauss-Markov and has
independent increments. However, empirical data from many physical
phenomena suggest that Brownian motion is often shown not to be an
effective process to use in a model. A family of processes that
seems to have wide physical applicability is fractional Brownian
motion (fBm). This process was introduced by Kolmogorov in
\cite{kolmo41} and later studied by Mandelbrot and Van Ness in
\cite{MV68}, where a stochastic integral representation in term of a
standard Brownian motion was obtained. Since the fBm  $B^H$ is not a
semimartingale if $H\neq\frac{1}{2}$ (see \cite{biagini08}), we can
not use the classical It\^{o} theory to construct a stochastic
calculus with respect to fBm.

Since some physical phenomena are naturally modeled by stochastic
partial differential equations or stochastic integro-differential
equations and the randomness can be described by a fBm, it is
important to study the controllability of infinite dimensional
equations with a fBm. Many studies of the solutions of stochastic
equations in an infinite dimensional space with a fBm have been
emerged recently, see
\cite{biagini08,boufoussi3,boufoussi2,carab,carab14,
lak15,Neuenkirch}. The literature related to  neutral stochastic
partial functional integro-differential equations driven by
    a fBm is not vast.     Very recently, in  \cite{carab13}, the authors studied the existence   and
uniqueness    of mild solutions for   a class of
   stochastic delay partial functional  integro-differential equations
     by  using the theory   of resolvent operators.

The control problems for stochastic equations driven by fractional
noise have been studied only recently and no results seem to be
available for  controllability. Motivated by \cite{carab13,shakt05},
 but the analysis for
fBm requires additional results and we have to construct a new
control, moreover,   we study the   controllability
 for the following neutral
   stochastic delay partial functional  integro-differential equations perturbed by a
   fractional Brownian motion:

  \begin{equation}\label{eq1}
 \left\{\begin{array}{llll}
d[x(t)+G(t,x(t-r(t)))]&&=[Ax(t)+G(t,x(t-r(t)))+Hu(t)]dt\\
\\ &&+\int_0^tB(t-s)[x(s) +G(s,x(s-r(s)))]dsdt\\
\\ &&+F(t,x(t-\rho(t)))dt +\sigma (t)dB^H(t),\;0\leq t
\leq T,  \\
 x(t)=\varphi (t), \; -\tau \leq t \leq 0.  
\end{array}\right.
 \end{equation}

\ni Here $A$ is the infinitesimal generator of a strongly continuous
semigroup of bounded linear operators, $(S(t))_{t\geq 0}$, in a
Hilbert space $X$ with domain $D(A)$, $B(t)$ is a closed linear
operator on $X$ with domain $D(B(t))\supset D(A) $ which independent
of $t$.  The control function $u(.)$ takes values in $L^2([0,T],U)$,
the Hilbert space of admissible control functions for a separable
Hilbert space  $U$. The symbol $H$ stands for a bounded linear
operator from $U$ into $X$.  $B^H$ is a  Fractional Brownian motion
on a real and separable Hilbert space $Y$, $r,\;
\rho:[0,+\infty)\rightarrow [0,\tau],\; (\tau
>0)
$
 are continuous and
 $
 F,G:[0,+\infty)\times X \rightarrow X,\;
\; \sigma:[0,+\infty) \rightarrow \mathcal{L}_2^0(Y,X) $
 are
appropriate functions. Here $\mathcal{L}_2^0(Y,X)$  denotes the
space of all $Q$-Hilbert-Schmidt operators from $Y$ into $X$ (see
section 2 below).

 In this paper, we study the controllability result with the help of resolvent operators. The resolvent operator is similar to the
 evolution operator for nonautonomous differential equations in a Hilbert spaces. It will not, however, be an evolution
 operator because it will not satisfy an evolution or semigroup
 property.   On the other hand, to the best of our
knowledge, there is no paper which investigates the controllability
of neutral stochastic integro-differential equations with delays
driven by a
 fractional Brownian motion . Thus, we will make the first attempt to study
such problem in this paper.

The  rest of this paper is organized as follows, In Section 2, we
introduce some notations, concepts, and basic results about
 fractional Brownian motion, Wiener integral over Hilbert spaces and we
 mention a few results and notations related with  resolvent of
 operators.  In Section 3, the controllability of the system (\ref{eq1}) is investigated via  a fixed-point analysis approach.
  Example presented in Section 4 demonstrates the controllability
  result of section 3.

\section{Preliminaries}

In this section, we recall some fundamental results needed to
establish our results. For details of this section, we refer the
reader to \cite{nualart,gri82} and
  references therein.\\

\subsection{ Fractional Brownian motion}

 Let $(\Omega,\mathcal{F},\{\mathcal{F}_t\}_{t\geq0}, \mathbb{P})$
be a complete probability space satisfying the usual condition,
which means that the filtration is right continuous increasing
family and $\mathcal{F}_0$ contains all P-null sets.

Consider a time interval $[0,T]$ with arbitrary fixed horizon $T$
and let $\{\beta^H(t) , t \in [0, T ]\}$ the one-dimensional
fractional Brownian motion with Hurst parameter $H\in(1/2,1)$. This
means by definition that $\beta^H$ is a centered Gaussian process
with covariance function:
$$ R_H(s, t) =\frac{1}{2}(t^{2H} + s^{2H}-|t-s|^{2H}).$$
 Moreover $\beta^H$ has the following Wiener
integral representation:
\begin{equation}\label{rep}
\beta^H(t) =\int_0^tK_H(t,s)d\beta(s),
 \end{equation}
where $\beta = \{\beta(t) :\; t\in [0,T]\}$ is a Wiener process, and
$K_H(t; s)$ is the kernel given by
$$K_H(t, s )=c_Hs^{\frac{1}{2}-H}\int_s^t (u-s)^{H-\frac{3}{2}}u^{H-\frac{1}{2}}du,
$$
for $t>s$, where $c_H=\sqrt{\frac{H(2H-1)}{\beta (2-2H,H-\frac{1}{2})}}$ and $\beta(,)$ denotes the Beta function. We put $K_H(t, s ) =0$ if $t\leq s$.\\
We will denote by $\mathcal{H}$ the reproducing kernel Hilbert space
of the fBm. In fact $\mathcal{H}$ is the closure of set of indicator
functions $\{1_{[0;t]},  t\in[0,T]\}$ with respect to the scalar
product
$$\langle 1_{[0,t]},1_{[0,s]}\rangle _{\mathcal{H}}=R_H(t , s).$$
The mapping $1_{[0,t]}\rightarrow \beta^H(t)$
 can be extended to an isometry between $\mathcal{H}$
and the first  Wiener chaos and we will denote by $\beta^H(\varphi)$
the image of $\varphi$ by the previous isometry.

We recall that for $\psi,\varphi \in \mathcal{H}$ their scalar
product in $\mathcal{H}$ is given by
$$\langle \psi,\varphi\rangle _{\mathcal{H}}=H(2H-1)\int_0^T\int_0^T\psi(s)\varphi(t)|t-s|^{2H-2}dsdt.
$$
Let us consider the operator $K_H^*$ from $\mathcal{H}$ to
$L^2([0,T])$ defined by
$$(K_H^*\varphi)(s)=\int_s^T\varphi(r)\frac{\partial K}{\partial
r}(r,s)dr.
$$ We refer to \cite{nualart} for the proof of the fact
that $K_H^*$ is an isometry between $\mathcal{H}$ and $L^2([0,T])$.
Moreover for any $\varphi \in \mathcal{H}$, we have
$$\beta^H(\varphi)=\int_0^T(K_H^*\varphi)(t)d\beta(t).
$$

It follows from \cite{nualart} that the elements of $\mathcal{H}$
may be not functions but distributions of negative order. In order
to obtain a space of functions contained in $\mathcal{H}$, we
consider the linear space $|\mathcal{H}|$ generated by the
measurable functions $\psi$ such that
$$\|\psi \|^2_{|\mathcal{H}|}:= \alpha_H  \int_0^T \int_0^T|\psi(s)||\psi(t)| |s-t|^{2H-2}dsdt<\infty,
$$
where $\alpha_H = H(2H-1)$. We have the following Lemma  (see
\cite{nualart})
\begin{lemma}\label{lem1}
The space $|\mathcal{H}|$ is a Banach space with the norm
$\|\psi\|_{|\mathcal{H}|}$ and we have the following inclusions
$$\mathbb{L}^2([0,T])\subseteq \mathbb{L}^{1/H}([0,T])\subseteq |\mathcal{H}|\subseteq \mathcal{H},
$$
and for any $\varphi\in \mathbb{L}^2([0,T])$, we have
$$\|\psi\|^2_{|\mathcal{H}|}\leq 2HT^{2H-1}\int_0^T
|\psi(s)|^2ds.
$$
\end{lemma}
Let $X$ and $Y$ be two real, separable Hilbert spaces and let
$\mathcal{L}(Y,X)$ be the space of bounded linear operator from $Y$
to $X$. For the sake of convenience, we shall use the same notation
to denote the norms in $X,Y$ and $\mathcal{L}(Y,X)$. Let $Q\in
\mathcal{L}(Y,Y)$ be an operator defined by $Qe_n=\lambda_n e_n$
with finite trace
 $trQ=\sum_{n=1}^{\infty}\lambda_n<\infty$. where $\lambda_n \geq 0 \; (n=1,2...)$ are non-negative
  real numbers and $\{e_n\}\;(n=1,2...)$ is a complete orthonormal basis in $Y$.
 Let $B^H=(B^H(t))$  be  $Y-$ valued fbm on
  $(\Omega,\mathcal{F}, \mathbb{P})$ with covariance $Q$ as
 $$B^H(t)=B^H_Q(t)=\sum_{n=1}^{\infty}\sqrt{\lambda_n}e_n\beta_n^H(t),
 $$
 where $\beta_n^H$ are real, independent fBm's. This process is  Gaussian, it
 starts from $0$, has zero mean and covariance:
 $$E\langle B^H(t),x\rangle\langle B^H(s),y\rangle=R(s,t)\langle Q(x),y\rangle \;\; \mbox{for all}\; x,y \in Y \;\mbox {and}\;  t,s \in [0,T].
 $$
In order to define Wiener integrals with respect to the $Q$-fBm, we
introduce the space $\mathcal{L}_2^0:=\mathcal{L}_2^0(Y,X)$  of all
$Q$-Hilbert-Schmidt operators $\psi:Y\rightarrow X$. We recall that
$\psi \in \mathcal{L}(Y,X)$ is called a $Q$-Hilbert-Schmidt
operator, if
$$  \|\psi\|_{\mathcal{L}_2^0}^2:=\sum_{n=1}^{\infty}\|\sqrt{\lambda_n}\psi e_n\|^2 <\infty,
$$
and that the space $\mathcal{L}_2^0$ equipped with the inner product
$\langle \varphi,\psi \rangle_{\mathcal{L}_2^0}=\sum_{n=1}^{\infty}\langle \varphi e_n,\psi e_n\rangle$ is a separable Hilbert space.\\

Now, let $\phi(s);\,s\in [0,T]$ be a function with values in
$\mathcal{L}_2^0(Y,X)$,
 The Wiener integral of $\phi$ with respect to $B^H$ is defined by

\begin{equation}\label{int}
\int_0^t\phi(s)dB^H(s)=\sum_{n=1}^{\infty}\int_0^t
\sqrt{\lambda_n}\phi(s)e_nd\beta^H_n(s)=\sum_{n=1}^{\infty}\int_0^t
 \sqrt{\lambda_n}(K_H^*(\phi e_n)(s)d\beta_n(s),
\end{equation}
where $\beta_n$ is the standard Brownian motion used to  present $\beta_n^H$ as in $(\ref{rep})$.\\
Now, we end this subsection by stating the following result which is
fundamental to prove our result. It can be proved by  similar
arguments as those used to prove   Lemma 2 in \cite{carab}.
\begin{lemma}\label{lem2}
If $\psi:[0,T]\rightarrow \mathcal{L}_2^0(Y,X)$ satisfies $\int_0^T
\|\psi(s)\|^2_{\mathcal{L}_2^0}ds<\infty$
 then the above sum in $(\ref{int})$ is well defined as a $X$-valued random variable and
 we have$$ \mathbb{E}\|\int_0^t\psi(s)dB^H(s)\|^2\leq 2Ht^{2H-1}\int_0^t \|\psi(s)\|_{\mathcal{L}_2^0}^2ds.
 $$
\end{lemma}

\subsection{Partial Integro-differential Equations}
For better comprehension of the
 subject we shall introduce some  definitions,  hypothesis and
 results. We  refer the reader to \cite{gri82}.
Throughout the rest of the paper we always assume that $X$ is a
Banach space, $A$  and $B(t)$ are closed linear operators on $X$.
$Y$ represents the Banach space  $D(A)$ equipped with the graph
norm defined by
$$
\| y\|_Y= \|Ay\|+\|y\|,\qquad \text{ for }y\in Y.
$$

 We consider the
following the abstract integro-differential problem
\begin{gather}
\frac{dx(t)}{dt}  = Ax(t)+ \int_{0}^{t} B(t-s)x(s) \, ds, \label{eqa1} \\
x(0)  =  x \in X. \label{eqa2}
\end{gather}

\begin{definition} \label{D3} \rm
 A  one-parameter family of bounded linear operators
$(R(t))_{t\geq 0} $  on $X$  is called a resolvent operator  of
\eqref{eqa1}-\eqref{eqa2} if the following conditions are satisfied.
\begin{itemize}
\item[(a)] $R(\cdot): [0, \infty) \to \mathcal{L}(X)$
 is strongly continuous and $R(0)x=x$ for all $x\in X$.

\item[(b)] For  $x \in D(A) $,  $R(\cdot)x \in C([0,\infty), [D(A)])
\cap C^{1}([0,\infty),X)$, and
\begin{gather}\label{eqrp1}
\frac{d R(t)x}{dt}
= A R(t)x + \int_{0}^{t} B(t-s) R (s) x d s, \\
\label{eqrp2} \frac{d R(t)x}{dt}  =  R(t) A x + \int_{0}^{t} R(t-s)
B(s)x d s,
\end{gather}
for every  $t\geq 0 $,
\item[(c)] There exists  some  constants $M>0,\delta$ such that
$\|R(t)\| \leq M e^{\delta t}$ for every $ t\geq 0$.
\end{itemize}
\end{definition}

\begin{definition} \label{def3} \rm
A resolvent operator $(R(t))_{t\geq 0}$ of \eqref{eqa1}-\eqref{eqa2}
 is called exponentially stable if there exists
positive constants $M, \alpha$ such that $\| R(t) \| \leq M e^{-
\alpha t}$.
\end{definition}
The resolvent operators play an important role to study the
existence of solutions and to give a variation of constants
formula for nonlinear systems. We need to know when the linear
system (\ref{eqa1})-(\ref{eqa2}) has a resolvent operator. For
more details on resolvent operators, we refer to
\cite{gri82,Grimmer4}. In this work we  assume that the following
conditions are satisfied:

\begin{itemize}
  \item [$(\mathcal{A}.1)$ ] $A$ is the infinitesimal generator of a strongly continuous semigroup on $X$.
  \item [ $(\mathcal{A}.2)$ ] For all $t \geq0$,  $B(t)$ is a closed linear operator from  $D(A)$ to $X$, and
   $B(t)\in\mathcal{L}(Y,X)$.
For any  $y \in Y$, the map $t \longrightarrow B(t)y$   is
bounded, differentiable and the derivative $B'(t)y$ is bounded and
uniformly continuous on $\R^+$ .\\
\end{itemize}
\begin{theorem} {\cite[Theorem 3.7 ]{gri82}}\label{thm:resol}
Assume that $(\mathcal{A}.1)$ and $(\mathcal{A}.2)$  hold. Then
there exists a unique resolvent operator of the Cauchy problem
(\ref{eqa1})-(\ref{eqa2}).

\end{theorem}

In the remaining  of this section  we discuss the  existence  of
solutions to
\begin{gather} \label{eqanhp1}
 \frac{dx(t)}{dt}
= Ax(t)+ \int_{0}^{t} B(t-s)x(s) \ d s + f(t), \quad   t  \in\geq0,     \\
\label{eqanhp2}    x(0) = z \in X,
\end{gather}
where $f: [0,+\infty)\longrightarrow X $ is a continuous function.
   We begin by introducing the following
concept of strict  solution.

\begin{definition}  \label{D1}
A function  $x : [0,+\infty) \to X $,   is called  a strict
solution of   \eqref{eqanhp1}-\eqref{eqanhp2} on $[0,+\infty) $ if
$ x \in C([0,+\infty),  [D(A)]) \cap C^{1}([0,+\infty), X) $, the
condition \eqref{eqanhp2} holds and the equation  \eqref{eqanhp1}
is satisfied on $[0,+\infty)$.
\end{definition}

\begin{theorem}[{\cite[Theorem 2.5]{gri82}}] \label{formvarconstnr}
Let $z\in X$. Assume that $f \in C([0,+\infty),X)$  and  $x(\cdot)
$ is a strict   solution of  \eqref{eqanhp1}-\eqref{eqanhp2} on
$[0, +\infty)$. Then
\begin{equation}  \label{formulavariation}
x(t) = R(t) z  +  \int_{0}^{t} R(t-s) f(s) \, d s,  \quad t \in
[0,+\infty).
\end{equation}
\end{theorem}

Motivated by  \eqref{formulavariation}, we introduce the following
concept.

\begin{definition}  \label{D2} \rm
A function $u \in C([0,+\infty),X) $ is called  a mild solution of
\eqref{eqanhp1}-\eqref{eqanhp2} if
$$
u(t) = R(t)z + \int_{0}^{t} R(t-s) f(s)\, d s, \quad t\in[0,T], \;
\text{ for } z\in X.
$$
\end{definition}

\section{Controllability Result}
In this section we study the controllability results  for Equation
(\ref{eq1}). Before starting, we introduce the concept of a mild
solution of the problem (\ref{eq1})  and controllability of
 neutral integro-differential stochastic functional differential equation.
 Motivated by the theory of resolvent operator, we introduce  the
following concept of mild solution for equation (\ref{eq1}).

\begin{definition}
An $X$-valued stochastic  process $\{x(t),\;t\in[-\tau,T]\}$, is
called a mild solution of equation (\ref{eq1}) if
\begin{itemize}
\item[$i)$] $x(.)\in \mathcal{C}([-\tau,T],\mathbb{L}^2(\Omega,X))$,
\item[$ii)$] $x(t)=\varphi(t), \, -\tau \leq t \leq 0$.
\item[$iii)$]For arbitrary $t \in [0,T]$, we have
\begin{equation}\label{mild}
\begin{array}{lll}
x(t)&=&R(t)(\varphi(0)+G(0,\varphi(-r(0))))-G(t,x(t-r(t)))\\
&+& \int_0^t R(t-s)[H u(s)+F(s,x(s-\rho (s)))]ds\\
&+&\int_0^t R(t-s)\sigma(s)dB^H(s)\;\;\; \mathbb{P}-a.s.
\end{array}
\end{equation}
\end{itemize}
\end{definition}

\begin{definition}
The  system (\ref{eq1}) is said to be controllable on the interval
$[-\tau,T]$, if for every initial stochastic process $\varphi$
defined $[-\tau,0]$ and  $x_1\in X$, there exists a stochastic
control $u\in L^2([0,T], U)$ such that the mild solution $x(.)$ of
(\ref{eq1}) satisfies $x(T)=x_1$.
\end{definition}

Roughly speaking, controllability problem for evolution system
consists in driving the state of the  system (the mild solution of
the controlled equation under consideration)  from an arbitrary
initial state to an arbitrary final state in finite time.

To prove the controllability result, we consider  the following
assumptions:
\begin{itemize}
\item [$(\mathcal{H}.1)$] The resolvent operator $(R(t))_{t\geq0}$ given by
$(\mathcal{A}.1)$ $(\mathcal{A}.2)$ satisfies the following
condition: there is a positive constant $M$ such that $$\sup_{0\leq
s,t\leq T} \|R(t-s)\|\leq M.$$

\item [$(\mathcal{H}.2)$]The function $f:[0,+\infty)\times X \rightarrow X$ satisfies the following Lipschitz
 conditions:
that is, there exist positive constants $C_i:=C_i(T), i=1,2$ such that, for all $t\in [0,T]$ and $x,y\in X $
\begin{itemize}
 \item[(i)] $\|F(t,x)-F(t,y)\|\leq C_1 \|x-y\|.$
\item[(ii)] $\|F(t,x)\|^2\leq C_2 (1+\|x\|^2).$
\end{itemize}
\item [$(\mathcal{H}.3)$] The function $G:[0,+\infty)\times X\longrightarrow X$   satisfies the following
conditions: there exist positive constants $C_3 $ and $C_4$,
$C_3<\frac{1}{2}$, such that,   for all $t\in [0,T]$ and $x,y\in X$
\begin{itemize}
 \item [(i)] $\|G(t,x)-G(t,y)\| \leq C_3 \|x-y\|.$
\item [(ii)] $\|G(t,x)\|^2 \leq C_4 (1+\|x\|^2).$
\end{itemize}
\item [$(\mathcal{H}.4)$]The function $G$ is continuous in the quadratic mean sense:
$$\mbox{For all}\;\; x\in \mathcal{C}([0,T], \mathbb{L}^2(\Omega, X)),\;\;
\lim_{t\rightarrow s}\mathbb{E}\|G(t,x(t))-G(s,x(s))\|^2=0.$$
\item [$(\mathcal{H}.5)$]The function $\sigma:[0,+\infty)\rightarrow \mathcal{L}_2^0(Y,X)$ satisfies
 $$\int_0^T\|\sigma(s)\|^2_{\mathcal{L}_2^0}ds< \infty,\;\; \forall T>0.
  $$
\end{itemize}
\begin{itemize}
\item [$(\mathcal{H}.6)$] The linear operator $W$ from $U$ into $X$
 defined by
 $$
Wu=\int_0^TR(T-s)H u(s)ds
 $$
 has an inverse operator $W^{-1}$ that takes values in $L^2([0,T],U)\setminus ker
 W$, where $ker
 W=\{x\in L^2([0,T],U),  \; W x=0\}$ (see \cite{klam07}), and there
 exists finite
 positive constants $M_b,$ $M_w$ such that $\|B\|\leq M_b$ and $\|W^{-1}\|\leq M_w.$
\end{itemize}
Moreover, we assume that $\varphi \in \mathcal{C}([-\tau,0],\mathbb{L}^2(\Omega,X))$.\\

We can now state the main result of this paper.
\begin{theorem}\label{th1}
Suppose that $(\mathcal{H}.1)-(\mathcal{H}.6)$ hold. Then,  the
system  (\ref{eq1}) is controllable  on $[-\tau,T]$.
\end{theorem}
\begin{proof}
Fix $T>0$ and let  $\mathcal{B}_T := \mathcal{C}([-\tau, T],
\mathbb{L}^2(\Omega, X))$ be the Banach space of all   continuous
functions from $[-\tau, T]$ into $\mathbb{L}^2(\Omega, X)$,
equipped  with the  supremum norm
$\|\xi\|_{\mathcal{B}_T}=\displaystyle\sup_{u \in
[-\tau,T]}\left(\mathbb{E} \|\xi (u)\|^2\right)^{1/2}$ and let us
consider the set
 $$S_T=\{x\in \mathcal{B}_T : x(s)=\varphi(s),\; \mbox {for} \;\;s \in [-\tau,0] \}.$$
 $S_T$ is a closed subset of $\mathcal{B}_T$ provided with the norm  $\|.\|_{\mathcal{B}_T}$.

Using the hypothesis (H6) for an arbitrary function $x(.)$, define
the control

\begin{equation}\label{control}
\begin{array}{lll}
  u(t) & =&W^{-1}\{x_1-R(T)(\varphi(0)+G(0,\varphi(-r(0))))+G(T,x(T-r(T)))\\ \\
   & -&\int_0^T R(T-s)F(s-\rho   (s))ds-\int_0^TR(T-s)\sigma(s)dB^H(s)\}(t).
   \end{array}
\end{equation}

We shall now show that when using this control, the operator  $\Phi$
defined  on  $S_T$ by  $\Phi(x)(t)=\varphi(t)$ for $t\in [-\tau,0]$
 and for $t\in [0,T]$
 \begin{equation}\label{operator}
 \begin{array}{lll}
 \Phi(x)(t)&=&R(t)(\varphi(0)+G(0,\varphi(-r(0))))-G(t,x(t-r(t)))\\
&+&\int_0^t R(t-s)[H u(s)+F(s-\rho
(s))]ds]+\int_0^tR(t-s)\sigma(s)dB^H(s)
 \end{array}
 \end{equation}

 has a fixed point. Substituting (\ref{control}) in (\ref{operator})
 we can show that $\psi x(T)=x_1$, which means that the
control $u$ steers the system from the initial state $\varphi$ to
$x_1$ in time $T$,  provided we can obtain a fixed point of the
operator $\psi$ which implies that the system in controllable.

 Next we will show by using Banach fixed point
theorem that $\psi$ has a unique fixed point. We divide the
subsequent proof into two steps.

Step 1: For arbitrary $x\in S_T$, let us prove that $t\rightarrow \Phi(x)(t)$ is continuous on the interval $[0, T]$ in the $\mathbb{L}^2(\Omega,X)$-sense.\\
Let $0 <t<T$  and $|h|$  be sufficiently small. Then for any fixed $x\in S_T$, we have
\begin{eqnarray*}
\E\|\Phi(x)(t&+&h)-\Phi(x)(t)\|^2\leq 5 \E\|(R(t+h)-R(t))(\varphi(0)+G(0,\varphi(-r(0))))\|\\
&+&5\E\|G(t+h,x(t+h-r(t+h)))-G(t,x(t-r(t)))\|\\
&+&5\E\|\int_0^{t+h} R(t+h-s)F(s,x(s-\rho (s)))ds-\int_0^t R(t-s)F(s,x(s-\rho (s)))ds\|^2\\
&+&5\E\|\int_0^{t+h}R(t+h-s)\sigma(s)dB^H(s)-\int_0^tR(t-s)\sigma(s)dB^H(s)\|\\
&+&5\E\|\int_0^{t+h}R(t+h-\nu)HW^{-1}\{x_1-R(T)(\varphi(0)+G(0,\varphi(-r(0))))\\
&+&G(T,x(T-r(T)))    -\int_0^T R(T-s)F(s,x(s-\rho (s)))ds\\ \\
&-&\int_0^T
   R(T-s)\sigma(s)dB^H(s)\}d\nu \\
   \\&-& \int_0^{t}R(t-\nu)HW^{-1}\{x_1-R(T)(\varphi(0)+G(0,\varphi(-r(0))))
+G(T,x(T-r(T)))   \\   \\
    & -&\int_0^T R(T-s)F(s,x(s-\rho (s)))ds -\int_0^T
   R(T-s)\sigma(s)dB^H(s)\}d\nu\\
   \\
 &=& \sum_{1\leq i \leq 5}5\E\|I_i(h)\|^2.
\end{eqnarray*}
We are going to show that each function $t\rightarrow I_i(t)$ is
continuous on $[0,T]$ in the $L^2$ sens.\\
 By the strong
continuity of $R(t)$, we have
$$\lim_{h\rightarrow 0}(R(t+h)-R(t))(\varphi(0)+G(0,\varphi(-r(0))))=0.$$
 The  condition $(\mathcal{H}.1)$ assures that $$\|(R(t+h)-R(t))
 (\varphi(0)+G(0,\varphi(-r(0))))\|\leq 2M \|\varphi(0)+G(0,\varphi(-r(0)))\| \in \mathbb{L}^2(\Omega).$$
 Then we conclude by the Lebesgue dominated theorem that $$\lim_{h\rightarrow
 0}\mathbb{E}\|I_1(h)\|^2=0.
 $$
 By using the fact that the operator $G$ is continuous in the quadratic mean sense,
 we conclude by condition $(\mathcal{H}.4)$ that
  $$\lim_{h\rightarrow 0}\mathbb{E}\|I_2(h)\|^2=0.
  $$

For the third term $I_3(h)$, we suppose that $h>0$ (Similar estimates hold for $h<0$), then we have
\begin{eqnarray*}
\|I_3(h)\|&\leq & \|\int_0^t (R(t+h-s)-R(t-s))F(s,x(s-r(s)))ds\|\\
&& +\|\int_t^{t+h} R(t+h-s)F(s,x(s-r(s)))ds\|\\
&\leq & I_{31}(h)+I_{32}(h).
\end{eqnarray*}
 By H\"older's inequality, one has that
$$\mathbb{E}|I_{31}(h)|^2\leq t\mathbb{E}\int_0^t \|(R(t+h-s)-R(t-s))F(s,x(s-r(s)))\|^2ds$$
By using the strong continuity of $R(t)$, we have for each $s \in
[0,t]$, $$\lim_{h\rightarrow 0}(R(t+h-s)-R(t-s)) F(s,x(s-r(s)))=0.
$$
By using condition $(\mathcal{H}.1)$, condition $(ii)$ in
$(\mathcal{H}.2)$,  we obtain
\begin{eqnarray*}
&&\|(R(t+h-s)-R(t-s))F(s,x(s-r(s)))\|^2\\
&& \;\;\;\leq 4M^2 \|F(s,x(s-r(s)))\|^2 ,
 \end{eqnarray*}

then, we conclude by the Lebesgue dominated theorem that
$$
\displaystyle\lim_{h\rightarrow 0}\mathbb{E}\|I_{31}(h)\|^2=0.
$$
 By conditions $(\mathcal{H}.1)$, $(\mathcal{H}.2)$  and H\"older's inequality, we get
$$\mathbb{E}\|I_{32}(h)\|^2 \leq  C_2^2hM^2\int_0^T  (\mathbb{E} \|x(s-r(s))\|^2+1)ds,$$
then  $$\lim_{h\rightarrow 0}\mathbb{E}\|I_3(h)\|^2 =0.
$$

For the term $I_4(h)$, we have
\begin{eqnarray*}
\|I_4(h)\|&\leq & \|\int_0^t (R(t+h-s)-R(t-s))\sigma(s)dB^H(s)\|\\
&+&\|\int_t^{t+h} R(t+h-s)\sigma(s)dB^H(s)\|\\
&\leq & I_{41}(h)+I_{42}(h).
\end{eqnarray*}

By condition $(\mathcal{H}.1)$ and  Lemma \ref{lem2}, we get that
\begin{eqnarray*}
E|I_{41}(h)|^2&\leq &2Ht^{2H-1}\int_0^t \|(R(t+h-s)-R(t-s))\sigma(s)\|_{\mathcal{L}_2^0}^2ds\\
\end{eqnarray*}
 Since $\displaystyle\lim_{h\rightarrow 0} \|(R(t+h-s)-R(t-s))\sigma(s)\|_{\mathcal{L}_2^0}^2=0$  and
 $$\|(R(t+h-s)-R(t-s))\sigma(s)\|_{\mathcal{L}_2^0}^2\leq 4 M^2 \|\sigma(s)\|_{\mathcal{L}_2^0}^2\in \mathbb{L}^1([0,T],ds),$$
 we conclude, by the dominated convergence theorem that,
 $$ \lim_{h\rightarrow 0}\mathbb{E}|I_{41}(h)|^2=0 .
  $$
 Again by Lemma \ref{lem2}, we get that
$$\mathbb{E}|I_{42}(h)|^2\leq 2Hh^{2H-1}M^2\int_t^{t+h} \|\sigma(s)\|_{\mathcal{L}_2^0}^2ds\rightarrow 0.
$$

\ni Next,  let's observe that

$$
\begin{array}{ll}
 \E\|I_{5}(h)\|^2 &\leq 2\E
\|
\int_t^{t+h}R(t+h-\nu)BW^{-1}\{x_1-R(T)(\varphi(0)+G(0,\varphi(-r(0))))\\\\
&+ G(T,x(T-r(T))) -\int_0^T R(T-s)F(s,x(s-\rho (s)))ds \\\\
&-\int_0^T
   R(T-s)\sigma(s)dB^H(s)\}d\nu\|^2\\\\
   &+2\E\|\int_0^{t}(R(t+h-\nu)-R(t-\nu))BW^{-1}\{x_1-R(T)(\varphi(0)+G(0,\varphi(-r(0))))\\\\
   &+G(T,x(T-r(T)))      -\int_0^T R(T-s)F(s,x(s-\rho (s)))ds\\ \\
   &-\int_0^T
   R(T-s)\sigma(s)dB^H(s) \}d\nu\|^2
   \\\\
   &\leq 2[\E\|I_{5,1}(h)\|^2+\E\|I_{5,2}(h)\|^2].
\end{array}
$$

 Let's first deal with $I_{5,1}(h)$, using  conditions
 $(\mathcal{H}.1)-(\mathcal{H}.6)$ and H\"{o}lder inequality, it follows that

$$
\begin{array}{ll}
  \E\|I_{5,1}(h)\|^2 &\leq5 M^2M_b^2M_w^2\int_t^{t+h}\{\E\|x_1\|^2+M^2\E\|\varphi(0)+G(0,\varphi(-r(0)))\|^2\\\\
  &+C_4^2(1+\sup_{s\in[-\tau,T]}\E\|x(s)\|^2) +M^2TC_2^2(1+\sup_{s\in[-\tau,T]}\E\|x(s)\|^2)\\ \\
  &+2M^2HT^{2H-1}\int_0^T\|\sigma(s)\|_{\mathcal{L}_2^0}^2ds\}d\nu.
\end{array}
$$

It results that
$$
 \lim_{h\rightarrow 0}\mathbb{E}||I_{5,1}(h)||^2=0.
$$
 In a similar way, we have
$$
\begin{array}{ll}
  \E\|I_{5,2}(h)\|^2 & \leq5 M_b^2M_w^2\int_0^{t}\|(R(t+h-\nu)-R(t-\nu))\|^2\{\E\|x_1\|^2\\ \\
  &+M^2\E\|\varphi(0)+G(0,\varphi(-r(0)))\|^2
  + C_4^2(1+\sup_{s\in[-\tau,T]}\E\|x(s)\|^2)\\\\
     & +M^2
   T^2C_2^2(1+\sup_{s\in[-\tau,T]}\E\|x(s)\|^2)+2M^2HT^{2H-1}\int_0^T\|\sigma(s)\|_{\mathcal{L}_2^0}^2ds\}d\nu.
\end{array}
$$
Since
$$
\begin{array}{ll}
 & \|R(t+h-\nu)-R(t-\nu)\|^2\{\E\|x_1\|^2+M^2\E\|\varphi(0)+G(0,\varphi(-r(0)))\|^2
  \\ \\
  &+ C_4^2(1+\sup_{s\in[-\tau,T]}\E\|x(s)\|^2)+M^2T^2C_2^2(1+\sup_{s\in[-\tau,T]}\E\|x(s)\|^2)\\ \\
  & +2M^2HT^{2H-1}\int_0^T\|\sigma(s)\|_{\mathcal{L}_2^0}^2ds\}\\ \\
 &  \leq  4M^2
   \{\E\|x_1\|^2+M^2\E\|\varphi(0)+G(0,\varphi(-r(0)))\|^2+ C_4^2(1+\sup_{s\in[-\tau,T]}\E\|x(s)\|^2)\\ \\
  & +M^2
   T^2C_2^2(1+\sup_{s\in[-\tau,T]}\E\|x(s)\|^2)+2M^2HT^{2H-1}\int_0^T\|\sigma(s)\|_{\mathcal{L}_2^0}^2ds\}  \in
  L^1([0,T],ds]),
\end{array}
$$
we conclude, by the dominated  convergence  theorem that,
$$
\lim_{h\rightarrow 0}\mathbb{E}||I_{5,2}(h)||^2=0.
$$

The above arguments show that $\displaystyle\lim_{h\rightarrow
0}\mathbb{E}\|\Phi(x)(t+h)-\Phi(x)(t)\|^2=0$.
 Hence, we conclude that  the function  $t \rightarrow \Phi(x)(t)$ is continuous on $[0,T]$
 in the $\mathbb{L}^2$-sense.\\

\ni Step 2.  Now, we are going to show that $\Phi$ is a
contraction mapping in $S_{T_1}$ with some $T_1\leq T$ to be
specified later.

 Let $x,y\in S_T$  we obtain for any fixed  $t\in [0,T]$
\begin{eqnarray*}
\|\Phi(x)(t)&-&\Phi(y)(t)\|^2\\
&\leq&4\|G(t,x(t-r(t)))-G(t,y(t-r(t)))\|^2\\ \\
&&+4\|\int_0^tR(t-s)[F(s,x(s-\rho(s)))-F(s,y(s-\rho(s)))]ds\|^2\\ \\
&&+4\|
\int_0^{t}R(t-\nu)BW^{-1}[G(T,x(T-r(T)))-G(T,y(T-r(T)))]d\nu\|^2\\
\\
 &&+ 4\|\int_0^{t}R(t-\nu)BW^{-1}\int_0^T R(T-s)[F(s,x(s-\rho
 (s)))-F(s,y(s-\rho (s)))]dsd\nu\|^2.
 \end{eqnarray*}
By Lipschitz property of $ F$ and $ G$ combined with  H\"older's
inequality, we obtain
\begin{eqnarray*}
\mathbb{E}\|\Phi(x)(t)-\Phi(y)(t)\|^2 &\leq &4 C_3^2
\mathbb{E}\|x(t-r(t))-y(t-r(t))\|^2\\
&&+  4M^2C_1^2t
\int_0^t\mathbb{E}\|x(s-r(s))-y(s-r(s))\|^2ds\\
&&+4\bf{t}
M^2M_b^2M_w^2[\E\|x(T-r(T))-y(T-r(T))\|^2\\\\
&&+T^2C_1^2M^2 \sup_{s\in[-\tau,t]} \mathbb{E}\|x(s)-y(s)\|^2.
\end{eqnarray*}
Hence
$$\sup_{s\in[-\tau,t]}\mathbb{E}\|\Phi(x)(s)-\Phi(y)(s)\|^2\leq
\gamma(t) \sup_{s\in[-\tau,t]} \mathbb{E}\|x(s)-y(s)\|^2.$$ where
$$ \gamma(t)=4[C_3^2+M^2C_1^2t^2+\bf{t}
M^2M_b^2M_w^2(1+T^2C_1^2M^2)].
$$

 By condition $(iii)$ in  $(\mathcal{H}.3)$, we have
$\gamma(0)=4  C_3^2 <1$. Then there exists $0<T_1\leq T $ such that
$0<\gamma(T_1)<1$ and $\Phi$ is a contraction mapping on $S_{T_1}$
and therefore has a unique fixed point, which is a mild solution of
equation (\ref{eq1}) on $[-\tau,T_1]$. This procedure can be
repeated in order to extend the solution to the entire interval
$[-\tau,T]$ in finitely  many steps. Clearly, $(\psi x)(T)=x_1$
which implies that the system (\ref{eq1}) is controllable on
$[-\tau,T]$.    This completes the proof.
\end{proof}

\section{Example}

 We consider the following   stochastic partial neutral functional integro-differential
  equation with  finite delays $\tau_1$ and $\tau_2$ $(0\leq
\tau_i\leq\tau<\infty,\; i=1,2)$, driven by a    fractional Brownian
motion of the form
   \small{
   \begin{equation}\label{eq:integrodiff}
 \left\{\begin{array}{llll}
 \frac{\partial}{\partial t
}[x(t,\xi)+g(t,x(t-\tau_1,\xi))]=\frac{\partial^2}{\partial^2\xi}
[x(t,\xi)+g(t,x(t-\tau_1,\xi))]\\ \\
+\int_0^tb(t-s)\frac{\partial^2}{\partial^2\xi}[x(s,\xi)+g(s,x(s-\tau_1,\xi))]ds \\ \\
+f(t,x(t-\tau_2,\xi))+\mu(t,\xi)+\sigma (t)\frac{dB^H}{dt}(t),\qquad t\geq0,  \\ \\
x(t,0)+g(t,x(t-\tau_1,0))=0,\qquad t\geq0, \\  \\
x(t,\pi)+g(t,x(t-\tau_1,\pi))=0,\quad t\geq0,  \\  \\

x(s,\xi)=\varphi(s,\xi) ,\,\;-\tau \leq s \leq 0\quad a.s.,
\end{array}\right.
\end{equation}
} where  $B^H(t)$ is   a    fractional Brownian motion, $f$,
$g:\R^+\times\R\longrightarrow\R$ are continuous functions and
$b:\R^+\longrightarrow\R$ is continuous function and $\varphi:
[-\tau,0]\times[0,\pi]\longrightarrow\R$ is a given continuous
function such that $\varphi(s,.)\in L^2([0,\pi])$ is measurable and
satisfies $\E\|\varphi\|^2<\infty.$

We rewrite (\ref{eq:integrodiff}) into abstract form of
(\ref{eq1}),  let $X=L^2([0,\pi])$. Define the operator
$A:D(A)\subset X\longrightarrow X$ given  by
$A=\frac{\partial^2}{\partial^2\xi}$ with domain
$$
D(A)= H^2([0,\pi])\cap H^1_0([0,\pi]),
$$
then we get
$$
Ax=\sum_{n=1}^\infty n^2<x,e_n>_Xe_n,\quad x\in D(A),
$$
 where $
e_n:=\sqrt{\frac{2}{\pi}}\sin nx,\; n=1,2,.... $
 is  an orthogonal  set of eigenvector of $-A$.\\

 It is well known that $A$ is the infinitesimal generator of a
 strongly  continuous  semigroup  of bounded linear operators $\{S(t)\}_{t\geq 0}$ in $X$,  thus $(\mathcal{A}.1)$ is true.
 Furthermore, $\{S(t)\}_{t\geq 0}$    is given by
(see \cite{pazy})
$$ S(t)x=\sum_{n=1}^{\infty}e^{-n^2t}<x,e_n>e_n
$$
for $x\in X$ and $t\geq0$,  that satisfies  $\|S(t)\|\leq
e^{-\pi^2t}$ for every $t\geq0$. \\
Let  $B:D(A)\subset X\longrightarrow X $ be the operator given by
$$
B(t)z=b(t)Az\qquad\text{ for } t\geq0 \;\text{ and  } z\in D(A).
  $$

 We assume that the following conditions
hold:
\begin{itemize}
  \item [(i) ]  Let  $Hu: [0,T]\longrightarrow X$  be  defined by
  $$
Hu(t)(\xi)=\mu(t,\xi),\;0\leq \xi \leq\pi, \, u\in L^2([0,T], U).
  $$
  \item [(ii) ]  Assume  that  the operator $W:L^2([0,T], U)\longrightarrow X $ given  by
  $$
W u(\xi)=\int_0^TR(T-s)\mu(t,\xi)ds,\;\;0\leq \xi \leq\pi,
  $$
  has  a bounded  invertible operator $W^{-1}$  and satisfies condition
  $(\mathcal{H}.6)$. For the construction of the operator $W$ and
  its inverse, see \cite{qui85}.


  \item [(iii) ] for $t\in[0,T]$, $f(t,0)=g(t,0)=0,$
  \item [(iv) ] there exist    positive constants $C_1$, and  $C_3$, $ C_3<\frac{1}{2}$,  such that
  $$
|f(t,\xi_1)-f(t,\xi_2)|\leq C_1|\xi_1-\xi_2|, \text{ for } t\in[0,T]
\text{ and } \xi_1, \xi_2 \in \R,
  $$
$$
|g(t,\xi_1)-g(t,\xi_2)|\leq C_3|\xi_1-\xi_2|, \text{ for } t\in[0,T]
\text{ and }  \; \xi_1, \xi_2 \in \R,
$$
  \item [(v) ] there exist   positive constants  $C_2$ and $C_4$, such that
  $$
|f(t,\xi) |\leq C_2(1+|\xi |^2), \text{ for } t\in[0,T] \text{ and }
\xi \in \R,
  $$
$$
|g(t,\xi) |\leq C_4(1+|\xi |^2), \text{ for } t\in[0,T] \text{ and }
\xi \in \R,
  $$

  \item [(vi) ] the function $\sigma:[0,+\infty)\rightarrow \mathcal{L}_2^0(L^2([0,\pi]),L^2([0,\pi]))$ satisfies
 $$\int_0^T\|\sigma(s)\|^2_{\mathcal{L}_2^0}ds< \infty,\;\; \forall T>0.
  $$
\end{itemize}
Define the operators $F, G: \R^+\times L^2([0,\pi])\longrightarrow
L^2([0,\pi]) $ by
$$
F(t,\phi)(\xi)=f(t,\phi(-\tau_1)(\xi))\; \text{ for } \xi\in[0,\pi]
\text{ and } \phi\in L^2([0,\pi]),
$$
and
$$
G(t,\phi)(\xi)=g(t,\phi(-\tau_2)(\xi)),\;\text{ and } \phi\in
L^2([0,\pi])
$$

 If we put
\begin{equation}
 \left\{\begin{array}{ll}
x(t)(\zeta)=x(t,\zeta),\;t\in[0,T]\text{ and } \;\zeta\in[0,\pi]\\
x(t,\zeta)= \varphi(t,\zeta), \; t\in[-\tau,0] \text{ and }\;
\zeta\in[0,\pi],
 \end{array}\right.
\end{equation}

then,  the problem (\ref{eq:integrodiff}) can be written in the
abstract form

  \begin{equation*}
 \left\{\begin{array}{llll}
d[x(t)+G(t,x(t-r(t)))]&&=[Ax(t)+G(t,x(t-r(t)))]dt+\int_0^tB(t-s)[x(s) \\
\\ &&+G(s,x(s-r(s)))]dsdt+[F(t,x(t-\rho(t)))+Hu(t)]dt \\
\\ &&+\sigma (t)dB^H(t),\;0\leq t
\leq T,  \\
 x(t)=\varphi (t), \; -\tau \leq t \leq 0.  
\end{array}\right.
 \end{equation*}

Moreover,  if $b$ is bounded and $C^1$ function such that $b'$ is
bounded and uniformly continuous, then $(\mathcal{A}.1)$ and
$(\mathcal{A}.2)$ are satisfied and hence, by Theorem
\ref{thm:resol}, Equation (\ref{eq:integrodiff}) has a resolvent
operator $(R(t))_{t\geq0}$ on $X$. As a consequence of the
continuity of $f$ and $g$ and assumption (iii) it follows that $F$
and $G$ are continuous. By assumption (iv), one can see that
$$
\|F(t,\phi_1)-F(t,\phi_1)\|_{L^2([0,\pi])}\leq
C_1\|\phi_1-\phi_2\|_{L^2([0,\pi])},
$$
$$
\|G(t,\phi_1)-G(t,\phi_1)\|_{L^2([0,\pi])}\leq
C_3\|\phi_1-\phi_2\|_{L^2([0,\pi])}, \;\text{  with } \;
C_3<\frac{1}{2}.
$$

Furthermore, by assumption (v), it follows that

$$
\|F(t,\phi) \|\leq C_2(1+\|\phi \|^2),  \text{ for } t\in[0,T],
  $$
$$
\|G(t,\phi) \|\leq C_4(1+\|\phi \|^2),  \text{ for } t\in[0,T].
$$

then all the assumptions of Theorem \ref{th1} are fulfilled.
Therefore, we conclude that the system (\ref{eq:integrodiff}) is
controllable on $[-\tau,T]$.


\end{document}